\documentclass[a4paper]{article}
\usepackage{latexsym,bm}
\usepackage{amssymb,amsmath,amsthm}
\usepackage[english,german]{babel}
\usepackage[colorlinks=true]{hyperref}
\usepackage{color}
\usepackage{graphicx}

\newtheorem{theorem}{Theorem}[section]
\newtheorem{lemma}[theorem]{Lemma}
\newtheorem{remark}[theorem]{Remark}
\newtheorem{proposition}[theorem]{Proposition}

\newtheorem{corollary}[theorem]{Corollary}
\numberwithin{equation}{section}
\hypersetup{linkcolor=blue,urlcolor=red,citecolor=red}

\title{Minimal time control for the heat equation  with multiple impulses}

\author{Ya Xin\thanks{School of
Science, Hebei University of Technology, Tianjin 300400, China.
e-mail: 202321102016@stu.hebut.edu.cn.} \quad and \quad Qishu Yan
\thanks{Corresponding author. School of Science, Hebei University of Technology, Tianjin  300400, China.
e-mail: yanqishu@whu.edu.cn.
This work was supported by the Shijiazhuang Science and Technology Bureau
 under grant 241791227A.}}

\begin{document}

\selectlanguage{English}
\date{}

\maketitle
\begin{abstract} This paper investigates a minimal time control problem for the heat equation with multiple impulse controls.
We first establish the maximum principles for this problem
and then prove the equivalence between the minimal time impulse control problem and its corresponding minimal norm impulse control problem.
Extending our analysis to the minimal time function itself, we examine its analytical properties (particularly continuity and monotonicity) with respect to the control constraint.
This leads to a notable discovery: despite being continuous and generally monotonic, the function may fail to be strictly decreasing.
\\


\end{abstract}

\section{Introduction}

The vast majority of literature on time optimal control problems for parabolic equations considers controls that act at every time instant, or continuously in time,
as seen in \cite{03,02,05,14,16,22,25,28, 35}.
As an important control strategy, impulse control has been successfully applied in diverse fields
such as chemical engineering, economics, and finance \cite{08,09,17,18,21,23,24, 26,30,32,34}. This approach is particularly valuable
for systems where continuous control is impractical or unsustainable.
Recently, the minimal time  control problem for parabolic equations with a single impulse control  was discussed; see, for example, \cite{09,Lin,29}.
However, the case of multi-impulse controls remains largely  unexplored. This gap motivates our study on the minimal time control problem for the heat equation with multiple impulse controls.

\subsection{Formulation of the Problem}

Here we introduce some notations used in this paper: $\Omega\subset\mathbb{R}^N(N\geq1)$\,is a bounded domain with a $C^2$ smooth boundary $\partial\Omega$; $\omega_k(k\geq 1)$ is an open and nonempty subset in $\Omega$ with its characteristic function $\chi_{\omega_k}$; $\{e^{\Delta t}\}_{t\geq0}$ is the analytic semigroup on $L^2(\Omega)$ generated by Laplacian operator $\Delta$ with its domain $H^2(\Omega)\cap H^1_0(\Omega)$; $B_r(0)$ denotes the closed ball in $L^2(\Omega)$, centered at $0$ and of radius $r>0$; $d_{B_r(0)}(\cdot)$ denotes the distance function to the set $B_r(0)$. Let $\{\tau_k\}_{k\geq 1}$ be a strictly  increasing sequence of positive numbers  with $\displaystyle\lim_{k\to\infty} \tau_k=+\infty$. Consider the following impulse controlled heat equation:
\begin{equation}\label{2.1}
\left\{\begin{array}{lll}
	\partial_t y-\Delta y=0 & \text { in } \Omega \times(0,+\infty)\backslash\{\tau_k\}_{k\geq 1}, \\
	y=0 & \text { on } \partial \Omega \times(0,+\infty), \\
	y(\tau_k)=y(\tau_k^{-})+\chi_{\omega_{k}}  u_k & \text { in } \Omega,\text{ for each }k\geq 1,   & \\
	y(0)=y_0 & \text { in } \Omega .
\end{array}\right.
\end{equation}
Here, $y_0\in L^2(\Omega)$ is a given function; $y(\tau_k^-):=\lim\limits_{t\rightarrow\tau_k^-}y(t)$ in $L^2(\Omega)$; at the impulse instant $\tau_k$, a control $u_k$ acts on the system through local operator $\chi_{\omega_k}(k\in\mathbb{N}^+)$. Write $y(\cdot;y_0,u)$ for the solution of the system \eqref{2.1} with $u=\big(u_j\big)_{j\in\mathbb{N}^+}\in l^\infty\big(\mathbb{N}^+;L^2(\Omega)\big)$. It is well known that for each $k\geq 1$,
\begin{equation*}
y(t;y_0,u)=e^{\Delta t}y_0+\sum\limits_{j=1}^k e^{\Delta(t-\tau_j)}\chi_{\omega_j}u_j\,\,\text{ for any }\,\,t\in[\tau_k,\tau_{k+1}).
\end{equation*}
Throughout this paper,  the usual inner product and norm in $L^2(\Omega)$ are denoted by $\langle\cdot,\cdot\rangle$ and $\|\cdot\|$, respectively.  Define
\begin{equation*}
\|u\|_\infty:=\sup\limits_{k\in\mathbb{N}^+}\|u_k\|\,\,\text{ for any }\,\,u=\big(u_k\big)_{k\in\mathbb{N}^+}\in l^\infty\big(\mathbb{N}^+;L^2(\Omega)\big).
\end{equation*}
Given $M\geq0$, we denote the constraint set of  controls by
\begin{equation*}
\mathcal{U}_M:=\Big\{u\in l^\infty\big(\mathbb{N}^+;L^2(\Omega)\big) :\,\|u\|_\infty\leq M\Big\}.
\end{equation*}
Since the impulse control system \eqref{2.1} is not null controllable (see \cite{23})
and its solution undergoes free decay, we therefore  adopt the ball $B_r(0)$
 as the target set.
Accordingly, the minimal time  control problem studied in this paper is formulated as follows:
\begin{center}
$({\rm TP})_M\,\,\,\,\,\,\,\,t^*(M):=\inf\big\{t:\,y(t;y_0,u)\in B_r(0),\,\,u\in\mathcal{U}_M\big\}$.
\end{center}
Regarding the problem $({\rm TP})_M$, we provide several definitions  in order:
\begin{itemize}
\par\item[(i)] We call $t^*(M)$ the optimal time.
\par \item[(ii)] We call $u\in\mathcal{U}_M$ an admissible control, if there exists $t>0$ so that $y(t;y_0,u)\in B_r(0)$.
\par \item[(iii)]We call $u^*=\big(u_k^*\big)_{k\in\mathbb{N}^+}\in \mathcal{U}_M$ an optimal control, if $y\big(t^*(M);y_0,u^*\big)\in B_r(0)$ and $u^*_j=0$ for each $j$ such that
 $\tau_j>t^*(M)$.
\par\item[(iv)] We call $\widetilde{u}$ an optimal control with minimal norm, if $\widetilde{u}$ is an optimal control and satisfies that $\|\widetilde{u}\|_\infty\leq\|u^*\|_\infty$\,for any optimal control $u^*$.
\end{itemize}
\vspace{0.3cm}

Let $$\gamma(y_0):=\min\big\{t>0:y(t;y_0,0)\in B_r(0)\big\}.$$
In this paper, we always assume $\gamma(y_0)>\tau_1$, as otherwise the time optimal control problem becomes trivial.
Denote by $k_0$ the maximum integer such that $\tau_{k_0}\leq \gamma(y_0)$. According to the definition of  problem $({\rm  TP})_M$,
it is clear that $$\tau_1\leq t^*(M)\leq \gamma(y_0) \textrm{ for any }M\geq 0.$$

\subsection{Main Results}

 The presence of multiple impulse controls in this problem introduces a major complication:
 the optimal time $t^*(M)$ may coincide with an impulse instant $\tau_k$.
 In such cases, the state  is discontinuous  at the optimal time and  the optimal controls may  be neither  bang-bang nor unique (see \cite{Huang}).
 To address this issue,  we employ  the concept of an optimal control with minimal norm rather than a general optimal control.
This paper first derives the
 maximum principle for an optimal control with minimal norm  when $t^*(M)$ coincides with an impulse instant.
 Based on this,  we next demonstrate  the equivalence  between the minimal time and minimal norm optimal control problems.
This equivalence, in turn, provides the necessary framework for  analyzing the dynamic behavior of the minimal time function $t^*(\cdot)$.

Pontryagin's maximum principle is an important topic for time optimal control problems. Its study dates
 back to the 1950s (see, for instance, \cite{40,41}), and
 it has been extensively developed by many mathematicians. Here, we refer to \cite{42,03,46,44,45,39}.
As for  time optimal control problems with impulse controls,
in \cite{09} the authors proved the maximum principle for a class of evolution system with a single impulse control,
by using  the Hahn-Banach theorem to separate the target set and the attainable set;
in \cite{29} the author  utilized the Ekeland variational principle and obtained the maximum principle for a semilinear heat equation, also with a single impulse control.
For the multi-impulse time optimal control problem studied here,
 the maximum principle requires distinguishing two distinct cases:
wether the optimal time $t^*(M)$ coincides with an impulse instant or not.
The case in which $t^*(M)$ does not coincide with any impulse instant was addressed in \cite{Huang}.
This paper  derives the maximum principle  for the ``optimal control with minimal norm" in the complementary case,
which also implies  the  bang-bang property and  uniqueness
 of the optimal control with minimal norm.

 \par Next, we study the corresponding minimal norm control problem and its relationship with the time  optimal control problem.
 We consider   two distinct cases based on the state at the first impulse instant:
\begin{itemize}
    \par\item[\;\;\;(C1)] \,\,\,$\big\{e^{\Delta\tau_1}y_0+\chi_{\omega_1}u_1:u_1\in L^2(\Omega)\big\}\cap B_r(0)=\varnothing.$
    \par\item[\;\;\;(C2)]\,\,\,\,\,\,$\big\{e^{\Delta\tau_1}y_0+\chi_{\omega_1}u_1:u_1\in L^2(\Omega)\big\}\cap B_r(0)\neq\varnothing.$
\end{itemize}
Suppose that (C1) stands. Let $T\in(\tau_1,\gamma(y_0)]$, the corresponding minimal norm control problem is defined as:
\begin{center}
$({\rm NP})_T\,\,\,\,\,\,\,\,N^*(T):=\inf\big\{\max\limits_{1\leq j\leq k}\|v_j\|:\,y(T;y_0,v)\in B_r(0),\,\,v=(v_1,\,v_2,...,v_k,\,0,\,0,...)\big\}$,
\end{center}
where $k$ is the maximum integer so that $\tau_k\leq T$.
Regarding the problem $({\rm NP})_T$, we provide several definitions:
\begin{itemize}
\item[ (i)] We call $N^*(T)$ the minimal norm.
\item[(ii)] We call $v=\big(v_j\big)_{j\in\mathbb{N}^+}$ an admissible control, if $y(T;y_0,v)\in B_r(0)$ and $v_j=0$ for each $j>k$.
\item[(iii)] We call $v^*$ an optimal control, if it is admissible and satisfies that $\|v^*\|_\infty=N^*(T)$.
\end{itemize}
The existence of optimal controls, along with the strict monotonicity and  continuity of the function $N^*(\cdot)$, is proved
(see Lemma \ref{lemma 3.2} and  Proposition  \ref{Proposition 3.1}--\ref{Proposition 3.2}). Then we
establish the equivalence between $({\rm TP})_M$ and $({\rm NP})_T$, as stated in the following theorem.
\vspace{0.2cm}
\begin{theorem}\label{Theorem 1.2}~Assume that (C1) is true.
\begin{itemize}
  \item[(1)]
       Let $T\in(\tau_1,\gamma(y_0)]$. It stands that
        \begin{equation}\label{2.6}
        t^*\big(N^*(T)\big)=T
        \end{equation}
        and any optimal control for $({\rm NP})_T$ is the optimal control for $({\rm TP})_{N^*(T)}$.
        \item[(2)] Let $M\in[0,+\infty)$.
        \begin{itemize}
              \item[(i)]
                If $t^*(M)\in(\tau_k,\tau_{k+1})$ for some $1\leq k\leq k_0$,  then it stands that
                \begin{equation}\label{2.7}
                N^*\big(t^*(M)\big)=M
                \end{equation}
                and the optimal control for $({\rm TP})_M$ is also an optimal control for $({\rm NP})_{t^*(M)}$.
              \item[(ii)]
                If $t^*(M)=\tau_k$ for some $2\leq k\leq k_0$,  then it stands that
                \begin{equation}\label{2.8}
                N^*\big(t^*(M)\big)=M_1,
                \end{equation}
                where $M_1$ is the minimal norm among the optimal controls for $({\rm TP})_M$, and the optimal control with minimal norm for $({\rm TP})_M$ is an optimal control for $({\rm NP})_{t^*(M)}$.
        \end{itemize}
\end{itemize}
\end{theorem}
Two notes about Theorem \ref{Theorem 1.2} are given:
\begin{itemize}

\item
The relationship between time optimal control and norm optimal control problems  governed by evolution equations with continuous-time control
is widely studied (see, for instance, \cite{38,48,49,51,50,35}; see also the books \cite{47,39}). This equivalence,   which usually holds between the two problems,
 plays an important role in the study of time optimal control problems, see \cite{48,25,37}.

\item The equivalence of these two kinds of problems for the
semilinear heat equation with a single impulse control was  studied in \cite{09,29}. It seems  that the equivalence for the heat equation with multiple impulse
controls has not been touched on.
\end{itemize}

Based on Theorem \ref{Theorem 1.2},
the analytical properties of the minimal time  function are further explored.
We find that this function is continuous and decreasing
but not necessarily strictly monotonic.
Specifically, it is strictly decreasing whenever the optimal time does not equal  any impulse instant of the system.
Conversely, when it coincides with an  impulse instant  the  function may {remain constant }over an  interval thereafter.

To be precise,  define for each $2\leq k\leq k_0$,
\begin{equation*}
M_k^{\inf}:=\inf\big\{M\geq 0:\, t^*(M)= \tau_k\big\}
\end{equation*}
and
\begin{equation*}
M_k^{\sup}:=\sup\big\{M\geq 0:\, t^*(M)= \tau_k\big\}.
\end{equation*}
These sets have been proved to be nonempty (see the beginning of Section 4). Then the entire  interval $[0,+\infty)$ is partitioned into a sequence of disjoint subintervals,  and  the following properties of $t^*(\cdot)$ are  verified.

\begin{theorem}\label{Theorem  4.4}~Assume that (C1) stands. The function $t^*(\cdot):\,[0,+\infty)\to (\tau_1,\gamma(y_0)]$ is continuous. In particular,
it keeps still over $[M_{k}^{\inf},M_{k}^{\sup}]$ for each $2\leq k\leq k_0$, and is strictly  decreasing over $[0,+\infty)\backslash \left(\cup_{k=2}^{k_0}\left[M_{k}^{\inf},M_{k}^{\sup}\right]\right)$. Moreover, it stands that
\begin{equation}\label{0717-15}
t^*(0)=\gamma(y_0) \textrm{ and }\lim_{M\to\infty} t^*(M)=\tau_1.
\end{equation}
\end{theorem}

Several notes are given in order:
\begin{itemize}
\item This plateau phenomenon for $t^*(\cdot)$ (see Figure 2 in Section 4) is, to our knowledge, novel and absent in both continuous-time and single-impulse control problems.
At the same time, the minimal norm function $N^*(\cdot)$ is also continuous and strictly decreasing over $(\tau_k,\tau_{k+1})$ for each $1\leq k\leq k_0-1$.
However,  it may be discontinuous at the impulse instants, which also appears to be new in  norm optimal control problems.
See Remark \ref{Remark 3.5} for more details.

\item The idea of ``a time optimal control with minimal norm" is borrowed from  \cite{37},
 where it was used  to study time optimal sampled-data control and its approximation error relative to the continuous-time case.
 This concept is crucial in the derivation of Theorem \ref{Theorem 1.2} and Theorem \ref{Theorem 4.4}.

\item
The results of this paper are derived under  condition (C1).
  Analogous conclusions under (C2) are presented in Remark \ref{Remark 3.6}.

\end{itemize}

\par The rest of the paper is organized as follows.
Section 2 establishes fundamental properties of problem $({\rm TP})_M$, including
the existence,  maximum principle,  bang-bang property and  uniqueness  of the optimal controls with minimal norm.
In Section 3, we first  examine key properties of the minimal norm function $N^*(\cdot)$
 and subsequently  prove Theorem \ref{Theorem 1.2}. Finally,
 Section 4 is devoted to the proof of Theorem \ref{Theorem 4.4} and further analysis of $N^*(\cdot)$.

\section{Fundamental properties of the problem \texorpdfstring{$(\mathrm{TP})_M$}{(TP)\_M}}
The existence of optimal controls---including those with  minimal norm---for problem $({\rm TP})_M$ is first present
in Lemma \ref{Theorem 2.1}. The proof follows arguments similar to those in Lemma 2.1 of \cite{09}; for completeness, we include it in the Appendix.
 \begin{lemma} \label{Theorem 2.1}~
 Let $M\geq 0$.   The problem $({\rm TP})_M$ admits an optimal control.
 Moreover, an optimal control with minimal norm exists for  problem $({\rm TP})_M$.
\end{lemma}

The following lemma, which pertains to problem $({\rm TP})_M$ for the case where $t^*(M)$ does not coincide with any impulse instant, is quoted from \cite{Huang}.
 \begin{lemma} \label{lemma 2.2}
 Let $M\geq0$ and  assume that $t^*(M)\in(\tau_k,\tau_{k+1})$ for some $1\leq k \leq k_0$.
  Let $u^*=\left(u_j^*\right)_{j\in\mathbb N^+}$ be an optimal control for problem $({\rm TP})_M$.
Then, $\|u_j^*\|=M$ for each $j=1,2,...,k$. Moreover, the optimal control is unique.
 \end{lemma}

Now we present   Pontryagin's maximum principle for
 $({\rm TP})_M$ for the  case where the optimal time $t^*(M)$
 coincides with an impulse instant.

\begin{theorem}\label{Theorem 1.1}~
Let $M\geq0$ and assume that $t^*(M)=\tau_k$ for some $1\leq k\leq k_0$.
Let $M_1$ be  the minimal norm among all optimal controls.
Then  $\widetilde{u}=\big(\widetilde{u}_j\big)_{j\in\mathbb{N}^+}$ is
an optimal control with minimal norm for problem $({\rm TP})_M$ if and only if  there exists a function
\begin{center}
$\varphi\in L^2\left(0,\tau_k;H_0^1\left(\Omega\right)\right)\bigcap \,W^{1,2}\left(0,\tau_k;H^{-1}\left(\Omega\right)\right)$
\end{center}
so that
\begin{equation}\label{2.5}
\left\langle \widetilde{u}_j,\chi_{\omega_j}\varphi\big(\tau_j\big)\right\rangle=\max\limits_{\|u_j\|\leq M_1}
\left\langle u_j,\chi_{\omega_j}\varphi\big(\tau_j\big)\right\rangle \textrm{ \,for each\, }1\leq j\leq k
\end{equation}
and
\begin{equation}\label{2.4}
\left\{\begin{array}{lll}
	\partial_t \varphi+\Delta\varphi=0 & \text { in } \Omega \times(0,\tau_k), \\
	\varphi=0 & \text { on } \partial \Omega \times(0,\tau_k) , \\
	\varphi(\tau_k)=-y(\tau_k; y_0, \widetilde u ) & \text { in } \Omega.
\end{array}\right.
\end{equation}

\end{theorem}

\begin{proof}
The necessity  follows from arguments analogous to  Theorem 2.3 in \cite{Huang}. We now  prove the sufficiency.
 Let $\widetilde{u}=\big(\widetilde{u}_j\big)_{j\in\mathbb{N}^+}$ be an optimal control with minimal norm for the problem $({\rm TP})_M$. Then $\max\limits_{1\leq j\leq k}\|\widetilde{u}_j\|=M_1\leq M$.  We set
\begin{equation*}
\mathcal{A}_{\tau_k}:=\Big\{e^{\Delta \tau_k}y_0+\sum\limits_{j=1}^ke^{\Delta(\tau_k-\tau_j)}\chi_{\omega_j}u_j:u=\big(u_j\big)_{j\in\mathbb{N}^+}\in\mathcal{U}_{M_1}\Big\}.
\end{equation*}
Obviously, $\mathcal{A}_{\tau_k}$ is a convex and closed subset of $L^2(\Omega)$. The rest of the proof will be carried out by four steps as follows.

\par {\itshape Step 1.}~We illustrate  that
\begin{equation}\label{0801-2}
y\big(\tau_k;y_0,\widetilde u\big)\in\mathcal{A}_{\tau_k}\cap \partial B_r(0).
\end{equation}

First, we  notice that $y\big(\tau_k;y_0,\widetilde u\big)\in\mathcal{A}_{\tau_k}\cap B_r(0).$
Seeking a contradiction,   suppose that $\|y(\tau_k;y_0,\widetilde{u})\|<r$.
Then we can find $\varepsilon>0$ small enough so that $\|y(\tau_k;y_0,\widetilde{u}-\varepsilon\widetilde{u})\|\leq r$.
This shows that $(1-\varepsilon)\widetilde u$ is an optimal control for $({\rm TP})_M$,
which contradicts  the fact that $\widetilde{u}$ is an optimal control with minimal norm. Thus \eqref{0801-2} is true.

\par {\itshape Step 2.}~We show that
\begin{equation}\label{5.1}
\mathcal{A}_{\tau_k}\cap B_r(0)=\{y(\tau_k;y_0,\widetilde{u})\}.
\end{equation}

 It suffices to show that $\mathcal{A}_{\tau_k}\cap B_r(0)$ has a unique element.
Seeking a contradiction,   suppose that $\mathcal{A}_{\tau_k}\cap B_r(0)$ contains another element (which is different from $y(\tau_k;y_0,\widetilde{u})$), denoted by
\begin{equation}\label{5.2}
\widehat{y}\big(\tau_k\big):=y\big(\tau_k;y_0,\widehat{u}\big)=e^{\Delta \tau_k}y_0+\sum\limits_{j=1}^ke^{\Delta(\tau_k-\tau_j)}\chi_{\omega_j}\widehat{u}_j
\end{equation}
with some $\widehat{u}=\big(\widehat{u}_j\big)_{j\in\mathbb{N}^+}\in\mathcal{U}_{M_1}$. Set $v:=(\widetilde{u}+\widehat{u})/2$. Clearly,
\begin{equation}\label{5.3}
y\big(\tau_k;y_0,v\big)=\big[y\big(\tau_k;y_0,\widetilde{u}\big)+y\big(\tau_k;y_0,\widehat{u}\big)\big]/2.
\end{equation}
Since $L^2(\Omega)$ is strictly convex, by \eqref{5.2} and the equality in \eqref{5.3},
we can find $v\in \mathcal U_{M_1}$ so that
\begin{equation*}
\big\|y\big(\tau_k;y_0,v\big)\big\|<r.
\end{equation*}
Then we can find $\varepsilon>0$ small enough so that $\|y(\tau_k;y_0,v-\varepsilon v)\|\leq r$. It is easy to see that
\begin{equation*}
\|v-\varepsilon v\|_\infty=(1-\varepsilon)\|v\|_\infty< M_1,
\end{equation*}
which leads to a contradiction with the fact that $\widetilde{u}$ is an optimal control with minimal norm for problem $({\rm TP})_M$. Thus, \eqref{5.1} is verified.
\par {\itshape Step 3.}~We prove \eqref{2.5}.

By \eqref{0801-2} and  \eqref{5.1},  we have  $\mathcal{A}_{\tau_k}\cap$ int\,$B_r(0)=\varnothing$.
Hence, according to the geometric version of the Hahn-Banach theorem (see, for instance, \cite{04}), there exist $\zeta_0\in L^2(\Omega)$ with $\zeta_0\neq 0$ and a constant $\iota$ so that
\begin{equation}\label{5.5}
\langle\zeta_0,z_2\rangle\leq \iota \leq \langle\zeta_0,z_1\rangle\,\,\text{ for all }\,\,z_1\in B_r(0)\,\,\text{ and }\,\,z_2\in\mathcal{A}_{\tau_k}.
\end{equation}
The inequality in above yields that
\begin{equation*}\label{5.7}
\left\langle\zeta_0,y\left(\tau_k;y_0,u\right)\right\rangle\leq \left\langle\zeta_0, y\left(\tau_k;y_0,\widetilde{u}\right)\right\rangle\,\,\text{ for all }\,\,u\in\mathcal{U}_{M_1},
\end{equation*}
i.e.,
\begin{equation}\label{100}
\sum\limits_{j=1}^k\left\langle \zeta_0,e^{\Delta(\tau_k-\tau_j)}\chi_{\omega_j}u_j\right\rangle\leq \sum\limits_{j=1}^k\left\langle \zeta_0,e^{\Delta(\tau_k-\tau_j)}\chi_{\omega_j}\widetilde{u}_j\right\rangle,\,\,\text{ for all }\,\,u=\left(u_j\right)_{j\in\mathbb{N}^+}\in\mathcal{U}_{M_1}.
\end{equation}

Let $\varphi$ be the unique solution to
\begin{equation}\label{0906-1}
\left\{\begin{array}{lll}
	\partial_t \varphi+\Delta\varphi=0 & \text { in } \Omega \times(0,\tau_k), \\
	\varphi=0 & \text { on } \partial \Omega \times(0,\tau_k) , \\
	\varphi(\tau_k)=\zeta_0 & \text { in } \Omega.
\end{array}\right.
\end{equation}
 From \eqref{100} and \eqref{0906-1}, we get that
\begin{equation*}\label{0906-3}
\sum\limits_{j=1}^k\left\langle \varphi(\tau_j),\chi_{\omega_j}u_j\right\rangle\leq \sum\limits_{j=1}^k\left\langle \varphi(\tau_j),\chi_{\omega_j}\widetilde{u}_j\right\rangle,\,\,\text{ for all }\,\,u=\left(u_j\right)_{j\in\mathbb{N}^+}\in\mathcal{U}_{M_1},
\end{equation*}
This proves \eqref{2.5}. Furthermore, combining this result with the unique continuation of the heat equation and the fact that $\zeta_0\neq 0$, shows that
\begin{equation}\label{0802-2}
\widetilde u_j=M_1\frac{\chi_{\omega_j}\varphi(\tau_j)}{\|\chi_{\omega_j}\varphi(\tau_j)\|}\,\,\text{ for each }\,\,1\leq j\leq k.
\end{equation}

{\it Step 4.} We conclude the proof.

By \eqref{5.5} we have that
\begin{equation*}
\left\langle\zeta_0,z_1-y\left(\tau_k;y_0,\widetilde u\right)\right\rangle\geq0\,\,\text{ for all }\,\,z_1\in B_r(0).
\end{equation*}
Then
\begin{equation*}
\left\langle \zeta_0,y\left(\tau_k;y_0,\widetilde u\right)\right\rangle =\min\limits_{z_1\in B_r(0)}\left\langle\zeta_0,z_1\right\rangle=-r\|\zeta_0\|.
\end{equation*}
This shows that
\begin{equation}\label{0906-4}
y\big(\tau_k;y_0,\widetilde u\big)=c\zeta_0\,\,\text{ for some }\,\,c<0.
\end{equation}
By \eqref{0906-4} and \eqref{0802-2} we can choose $\zeta_0=-y\big(\tau_k;y_0,\widetilde u\big)$. This, together with \eqref{0906-1},
shows that \eqref{2.4}  holds.

\par This completes the proof.
\end{proof}

\begin{corollary} \label{Corollary 2.4}~Let $M\geq 0$ and assume that $t^*(M)=\tau_k$ for some $1\leq k\leq k_0$.
Let $\widetilde u=\left(\widetilde u_j\right)_{j\in\mathbb N^+}$ be an optimal control with the minimal norm $M_1$ for problem $({\rm TP})_M$.
Then $\|\widetilde u_j\|=M_1$ \,for each $j=1,2,...,k$. Moreover, the optimal control with minimal norm for problem $({\rm TP})_M$ is unique.
\end{corollary}
\begin{proof}
Equation \eqref{0802-2} implies that $\widetilde u$ is bang-bang,
 while its  uniqueness can be proved by a standard argument  based on the bang-bang property and the Parallelogram Law (see, e.g., \cite{09}).
\end{proof}

\par The bang-bang property of the optimal control for $({\rm TP})_M$  also implies  monotonicity of $t^*(\cdot)$, as stated in the following proposition.

\begin{proposition}\label{Proposition 2.6}~Let $0\leq M_2<M_1<+\infty$. Then $t^*(M_2)\geq t^*(M_1)$.
Furthermore, if  both $t^*(M_1)$ and $t^*(M_2)$  lie in the same interval $(\tau_k,\tau_{k+1})$ for some $1\leq k\leq k_0$, then the inequality is strict, i.e.,
\begin{equation}\label{1.1}
t^*(M_2)>t^*(M_1).
\end{equation}
\end{proposition}

\begin{proof}For any $0\leq M_2<M_1<+\infty$, we have $\mathcal{U}_{M_2}\subset\mathcal{U}_{M_1}$. Then it follows that $t^*(M_2)\geq t^*(M_1)$.
\par Suppose that $t^*(M_1),t^*(M_2)\in(\tau_k,\tau_{k+1})$ for some $1\leq k\leq k_0$ and  \eqref{1.1} was not true. Then, we would have that
\begin{equation}\label{10.1}
t^*(M_2)=t^*(M_1).
\end{equation}
Denote the optimal control for $({\rm TP})_{M_2}$ by $u^2$. Then it satisfies that
\begin{equation}\label{1.2}
\big\|u^2\big\|_\infty\leq M_2\,\,\text{ and }\,\,y\big(t^*(M_2);y_0,u^2\big)\in B_r(0).
\end{equation}
By \eqref{10.1} and \eqref{1.2}, we get that
\begin{equation}\label{1.3}
\big\|u^2\big\|_\infty\leq M_2<M_1\,\,\text{and}\,\,y\big(t^*(M_1);y_0,u^2\big)\in B_r(0).
\end{equation}
This implies that $u^2$ is an optimal control for problem $({\rm TP})_{M_1}$.
Then by Lemma \ref{lemma 2.2}, we have $\|u^2\|_\infty=M_1$. This contradicts the first inequality in \eqref{1.3}.
Therefore \eqref{1.1} holds.
\end{proof}

\section{Equivalence Between \texorpdfstring{$({\rm TP})_M$}{TP\_M} and \texorpdfstring{$({\rm NP})_T$}{NP\_T}}
\subsection{Some properties on problem \texorpdfstring{$(\mathrm{NP})_T$}{(NP)\_T}}
First  we introduce the following lemma quoted from \cite{08}.
 \begin{lemma} \label{lemma 3.1}~The system \eqref{2.1} is approximately null controllable for any $T>\tau_1$.
 That is, for any $T>\tau_1$, $r>0$ and $y_0\in L^2(\Omega)$,
 there exists a control $u\in l^\infty(\mathbb N^+;L^2(\Omega))$ so that $y(T;y_0,u)\in B_r(0).$
\end{lemma}

\begin{lemma} \label{lemma 3.2}~Assume that (C1) stands.
For each $T\in(\tau_1,\gamma(y_0)]$, there exists at least one optimal control for problem $({\rm NP})_T$.
\end{lemma}

\begin{proof}~We assume that $T\in[\tau_k,\tau_{k+1})$ for some $1\leq k\leq k_0$. By Lemma \ref{lemma 3.1}, there exists a sequence $\big\{u^\ell\big\}_{\ell\geq 1}\subset l^\infty\big(\mathbb{N}^+;L^2(\Omega)\big)$, where $u^\ell=\big(u^{\ell}_j\big)_{j\in\mathbb{N}^+}$ and $u_j^\ell=0$ for each $j>k$, so that
\begin{equation}\label{8.1}
\lim\limits_{\ell\rightarrow +\infty}\|u^{\ell}\|_\infty=N^*(T)\,\,\,\text{ and }\,\,\,y(T;y_0,u^\ell)\in B_r(0).
\end{equation}
Moreover, there exists a control $\widetilde{u}=\big(\widetilde{u}_j\big)_{j\in\mathbb{N}^+}$ with $\widetilde{u}_j=0$ for each $j>k$, and a subsequence of $\big\{u^\ell\big\}_{\ell\geq 1}$, still denoted by itself, so that
\begin{equation}\label{8.2}
u^{\ell}_j
\,\xrightarrow{w}\,\widetilde{u}_j~~\text{in}~L^2(\Omega) \,\,\text{ for each }\,\,j=1,2,...,k.
\end{equation}
It follows from \eqref{8.2} and the second relation in \eqref{8.1} that
\begin{equation*}
y(T;y_0,{u}^\ell)\,\xrightarrow{w}\, y(T;y_0,\widetilde{u})\in B_r(0)\,\,\,\text{in}\,\,L^2(\Omega).
\end{equation*}
This yields that $\widetilde{u}$ is an admissible control for problem $({\rm NP})_T$. Furthermore, we have that
\begin{equation*}\label{8.3}
\|\widetilde{u}\|_\infty\leq\lim\limits_{\ell\rightarrow\infty}\|u^{\ell}\|_\infty=N^*(T).
\end{equation*}
According to the definition of $N^*(T)$, we obtain that $\|\widetilde{u}\|_\infty=N^*(T)$ ,
which yields that $\widetilde{u}$ is an optimal control for problem $({\rm NP})_T$.
Hence, there exists an optimal control for problem $({\rm NP})_T$.
\end{proof}
\vspace{0.2cm}
\par Since $\gamma(y_0)=\min\big\{t>0: \,y(t;y_0,0)\in B_r(0)\big\}$, we have
\begin{equation*}
N^*\big(\gamma(y_0)\big)=0.
\end{equation*}
Moreover, the minimal norm function satisfies the following propositions.

\begin{proposition}\label{Proposition 3.1}~Assume that (C1) stands.  If $\tau_1<T_1<T_2<\gamma(y_0)$,  then
\begin{equation*}
0< N^*(T_2)<N^*(T_1)<+\infty.
\end{equation*}
\end{proposition}

\begin{proof}
For any $0<T<\gamma(y_0)$, we have
\begin{equation*}
\|y(T;y_0,0)\|>r.
\end{equation*}
Therefore, $0$ is not an optimal control for $({\rm NP})_T$, which yields that $N^*(T)>0$.
This, together with  Lemma \ref{lemma 3.2}, shows  that
$$0<N^*(T)<+\infty \textrm{ for any } 0<T<\gamma(y_0).$$
\par Let $u^1=(u^1_1,u^1_2,...,u^1_{k_1},\,0,\,0,...)$ be an optimal control for $({\rm NP})_{T_1}$, where $k_1$ is the maximum integer so that $\tau_{k_1}\leq T_1$.
Then
\begin{equation}\label{5.15}
\|y(T_1;y_0,u^1)\|\leq r \,\,\text{ and }\,\,\|u^1\|_\infty=N^*(T_1).
\end{equation}
Denote by $k_2$ the maximum integer so that $\tau_{k_2}\leq T_2$. Since $T_1<T_2$, we have $k_1\leq k_2$.
Set $u^2=(u^2_j)_{j\in\mathbb N^+}$ with
\begin{equation}\label{5.16}
u^2_j=\left\{\begin{array}{ll}
	(1-\varepsilon) u_j^1 & 1\leq j\leq k_1, \\
	0 &j> k_1,
\end{array}\right.
\end{equation}
where  $0<\varepsilon <1$ will be determined later. Then we have
$
u^2_j=0 \textrm{ for each }j>k_2
$
and
\begin{equation}\label{6.12}
\begin{aligned}
y\big(T_2;y_0,u^2\big)&=e^{\Delta(T_2-T_1)}y\big(T_1;y_0,(1-\varepsilon)u^1\big)\\
&=e^{\Delta(T_2-T_1)}y\big(T_1;y_0,u^1\big)-\varepsilon e^{\Delta(T_2-T_1)}y\big(T_1;0,u^1\big).
\end{aligned}
\end{equation}
Since $\|e^{\Delta t}y\|\leq e^{-\lambda _1 t}\|y\|$ for all $t>0$ and $y\in L^2(\Omega)$, \eqref{6.12} shows that
\begin{equation*}
\|y(T_2;y_0,u^2)\|\leq e^{-\lambda_1(T_2-T_1)}\|y(T_1;y_0,u^1)\|+\varepsilon e^{-\lambda_1(T_2-T_1)}k_1\|u^1\|_\infty.
\end{equation*}
This, combined with \eqref{5.15}, indicates that
\begin{equation*}
\|y(T_2;y_0,u^2)\|\leq r-(1- e^{-\lambda_1(T_2-T_1)})r+\varepsilon e^{-\lambda_1(T_2-T_1)}k_1N^*(T_1).
\end{equation*}
Taking $\varepsilon$ small enough so that
\begin{equation*}
\varepsilon e^{-\lambda_1(T_2-T_1)}k_1N^*(T_1)\leq(1- e^{-\lambda_1(T_2-T_1)})r,
\end{equation*}
we get that
\begin{equation*}
\|y(T_2;y_0,u^2)\|\leq r.
\end{equation*}
This shows that $u^2$ is an admissible control for $({\rm NP})_{T_2}$. Then by \eqref{5.15} and  \eqref{5.16}, we obtain that
\begin{equation*}
N^*(T_2)\leq \|u^2\|_\infty<\|u^1\|_\infty=N^*(T_1).
\end{equation*}
This completes the proof.
\end{proof}

\vspace{0.2cm}
\begin{proposition}\label{Proposition 3.2}~Assume that (C1) stands. The function $N^*(\cdot)$ is right continuous over $\big(\tau_1,\gamma(y_0)\big)$ and left continuous over $\big(\tau_1,\gamma(y_0)\big]\backslash\{\tau_k\}_{k=1}^{k_0}$.
\end{proposition}

\begin{proof}The proof is divided into the following two steps.
\par {\itshape Step 1.}~We prove that $N^*(\cdot)$ is right continuous over $\left(\tau_1,\gamma(y_0)\right)$.
\par Let  $T\in (\tau_1,\gamma(y_0))$  be arbitrary but fixed. Then   $T\in[\tau_k,\tau_{k+1})$ for some   $1\leq k\leq k_0$.
Let $\{T_n\}_{n\geq 1}\subset(\tau_k,\tau_{k+1})\cap\left(\tau_1,\gamma(y_0)\right)$ be a sequence strictly  decreasing to $T$.
It suffices to prove that
\begin{equation}\label{6.16}
\lim\limits_{n\rightarrow \infty}N^*(T_n)= N^*(T).
\end{equation}

By the strict monotonicity of $N^*(\cdot)$, we have that
\begin{equation*}
N^*(T_1)<N^*(T_2)<\cdot\cdot\cdot<N^*(T_n)<\cdot\cdot\cdot<N^*(T).
\end{equation*}
This yields that
\begin{equation}\label{6.1}
\lim\limits_{n\rightarrow \infty}N^*(T_n)\leq N^*(T).
\end{equation}
Denote by $u^n=\big(u^n_j\big)_{j\in\mathbb{N}^+}$ an optimal control for $({\rm NP})_{T_n}$. Then it stands that
\begin{equation}\label{6.2}
\|u^n\|_\infty=N^*(T_n)\,\,\text{ and }\,\,\|y(T_n;y_0,u^n)\|\leq r.
\end{equation}
Therefore, there exists a subsequence of $\{u^n\}_{n\in\mathbb{N}^+}$, still denoted by itself, and a control $u^*=\big(u_j^*\big)_{j\in\mathbb{N}^+}$ such that
\begin{equation}\label{51}
u_j^n{\,\xrightarrow{w}\,}\,u_j^*~\text{ in }~L^2(\Omega)~~~\text{ for each }~1\leq j\leq k.
\end{equation}
It follows from \eqref{51} that
\begin{equation}\label{6.4}
\|u^*\|_\infty\leq\lim\limits_{n\rightarrow \infty}\|u^n\|_\infty\leq \lim\limits_{n\rightarrow \infty} N^*(T_n)
\end{equation}
and
\begin{equation}\label{0717-1}
y(T;y_0,u^n){\,\xrightarrow{w}\,} y(T;y_0,u^*) \textrm{ as } n\to\infty.
\end{equation}
Moreover, since $\displaystyle\lim_{n\to \infty} T_n=T$ and $\{T_n\}_{n\in\mathbb N^+}\subset (\tau_k, \tau_{k+1})$, it is easy to show that
\begin{equation}\label{0717-2}
\lim_{n\to \infty}\left[ y(T;y_0, u^n)-y(T_n;y_0,u^n)\right]=0.
\end{equation}

By writing
\begin{align*}\label{50}
\big\|y(T;y_0,u^*)\big\|^2=&\;\big\langle y(T;y_0,u^*)-y(T;y_0,u^n),y(T;y_0,u^*)\big\rangle\\
&+\big\langle y(T;y_0,u^n)-y(T_n;y_0,u^n),y(T;y_0,u^*)\big\rangle\\
&+\big\langle y(T_n;y_0,u^n),y(T;y_0,u^*)\big\rangle
\end{align*}
and using   \eqref{6.2},  we get that
\begin{align*}
\big\|y(T;y_0,u^*)\big\|^2\leq &\;\big\langle y(T;y_0,u^*)-y(T;y_0,u^n),y(T;y_0,u^*)\big\rangle\\
&+\big\langle y(T;y_0,u^n)-y(T_n;y_0,u^n),y(T;y_0,u^*)\big\rangle\\
&+r\|y(T;y_0,u^*)\|.
\end{align*}
Passing to the limit for $n\rightarrow \infty$ in the above inequality and  using  \eqref{0717-1}-\eqref{0717-2}, we obtain that
\begin{equation*}
\|y(T;y_0,u^*)\|\leq r.
\end{equation*}
This shows that $u^*$ is an admissible control for $({\rm NP})_T$ and then
$$
N^*(T)\leq\|u^*\|_\infty,
$$
which, together with \eqref{6.1} and \eqref{6.4}, shows that \eqref{6.16} stands.
\vspace{0.2em}

{\itshape Step 2.}~We prove that $N^*(\cdot)$ is left continuous over $(\tau_1,\gamma(y_0)]\backslash\{\tau_k\}_{k=1}^{k_0}$.
\par Let $T\in(\tau_1,\gamma(y_0)]\backslash\{\tau_k\}_{k=1}^{k_0}$ be arbitrary but fixed.
Without loss of generality, assume that $T\in(\tau_k,\tau_{k+1})$ for some $1\leq k\leq k_0-1$. Let $\{T_n\}_{n\geq 1}\subset(\tau_k,\tau_{k+1})$
be a sequence strictly  increasing to $T$.
It suffices to prove that
\begin{equation}\label{6.17}
\lim\limits_{n\rightarrow \infty}N^*(T_n)= N^*(T).
\end{equation}

By the strict monotonicity of $N^*(\cdot)$, we have that
\begin{equation*}\label{6.6}
N^*(T_1)>N^*(T_2)>\cdot\cdot\cdot>N^*(T_n)>\cdot\cdot\cdot>N^*(T).
\end{equation*}
Therefore, we have
\begin{equation*}
\lim\limits_{n\rightarrow\infty}N^*(T_n)\geq N^*(T).
\end{equation*}
By contradiction, suppose that
\begin{equation*}\label{6.7}
\lim\limits_{n\rightarrow\infty}N^*(T_n)=N^*(T)+2\delta~~\text{ for some }~~\delta>0.
\end{equation*}
Then we have
\begin{equation}\label{6.8}
N^*(T_n)>N^*(T)+2\delta~~\text{ for each }~~n\in\mathbb{N}^+
\end{equation}
and there exists $K\in\mathbb{N}^+$ so that
\begin{equation}\label{6.9}
N^*(T_n)<N^*(T)+3\delta~~\text{ for each }~~n\geq K.
\end{equation}
Denote by $u^n$ and $u^*$ an optimal control for $({\rm NP})_{T_n}$ and $({\rm NP})_T$, respectively. Then we have
\begin{equation}\label{6.13}
\|y(T_n;y_0,u^n)\|\leq r~~\text{ and }~~\|y(T;y_0,u^*)\|\leq r.
\end{equation}
Moreover, we have that
\begin{equation}\label{6.15}
\|{u^n}\|_\infty= N^*(T_n)~~\text{ and }~~\|{u^*}\|_\infty= N^*(T).
\end{equation}

Set
\begin{equation*}
\widehat{u}=\frac{2}{3}u^*+\frac{1}{3}u^K.
\end{equation*}
Then it follows from \eqref{6.9} and \eqref{6.15}  that
\begin{equation}\label{6.10}
\big\|\widehat{u}\big\|_\infty\leq\frac{2}{3}\big\|u^*\big\|_\infty+\frac{1}{3}\big\|u^K\big\|_\infty<\frac{2}{3} N^*\big(T\big)+\frac{1}{3}\big(N^*(T)+3\delta\big)<N^*\big(T\big)+\delta
\end{equation}
and
\begin{equation*}\label{6.11}
y\left(T;y_0,\widehat{u}\right)=\frac{2}{3}y(T;y_0,u^*)+\frac{1}{3}y\left(T;y_0,u^K\right).
\end{equation*}
This, together with \eqref{6.13}, shows that
\begin{equation*}
\begin{aligned}
\|y(T;y_0,\widehat{u})\|&\leq \frac{2}{3}\|y(T;y_0,u^*)\|+\frac{1}{3}\big\|e^{\Delta(T-T_K)}y\left(T_K;y_0,u^K\right)\big\|\\
&\leq \frac{2}{3} r+\frac{1}{3}e^{-\lambda_1(T-T_K)} r\\
&=r-\frac{r}{3}\left(1-e^{-\lambda_1(T-T_K)}\right)\\
&<r.
\end{aligned}
\end{equation*}
Since $y(\cdot;y_0,\widehat{u})$ is continuous over $(\tau_k,\tau_{k+1})$, there exists $n\in\mathbb N^+$ large enough  so that
\begin{equation*}
\|y(T_n;y_0,\widehat{u})\|\leq r.
\end{equation*}
This shows that $\widehat{u}$ is an admissible control for $({\rm NP})_{T_n}$. Then, by \eqref{6.10}, we get that
\begin{equation*}
N^*(T_n)\leq \|\widehat{u}\|_\infty<N^*(T)+\delta,
\end{equation*}
which leads to a contradiction with \eqref{6.8}. Therefore \eqref{6.17} is true.
\end{proof}

\vspace{0.2cm}
\subsection{Proof of Theorem \ref{Theorem 1.2}}
\par  We are ready to prove Theorem \ref{Theorem 1.2}. The proof is divided into 3 steps.
\par {\itshape Step 1.}~We prove \eqref{2.6}.
\par When $T=\gamma(y_0)$,  \eqref{2.6} holds trivially,
and the null control is   uniquely  optimal for both  $({\rm NP})_T$ and $({\rm TP})_{N^*(T)}$.

 Let $u^T$ be an optimal control for $({\rm NP})_T$ for a given $T\in\big(\tau_1,\gamma(y_0)\big)$. It follows that
\begin{equation}\label{4.1}
\|y(T;y_0,u^T)\|\leq r~~\text{ and }~~\|u^T\|_\infty=N^*(T)\geq 0.
\end{equation}
By \eqref{4.1} and the definition of problem $({\rm TP})_{N^*(T)}$, $u^T$ is an admissible control for $({\rm TP})_{N^*(T)}$ and then
\begin{equation}\label{11.1}
\tau_1<t^*\big(N^*(T)\big)\leq T.
\end{equation}
According to Lemma \ref{Theorem 2.1}, there  exists a control $\widehat{u}$ so that
\begin{equation*}
\|\widehat{u}\|_\infty\leq N^*(T)\,\,\text{ and }\,\,\big\|y\big(t^*(N^*(T));y_0,\widehat{u}\big)\big\|\leq r.
\end{equation*}
This shows that $\widehat{u}$ is an admissible control for problem $({\rm NP})_{t^*(N^*(T))}$, and then
\begin{equation}\label{4.2}
N^*\big(t^*\big(N^*(T)\big)\big)\leq \|\widehat{u}\|_\infty \leq  N^*\big(T\big).
\end{equation}
 By \eqref{4.2}  and the monotonicity of $N^*(\cdot)$ proved in Proposition  \ref{Proposition 3.1},
 we obtain  that $t^*\big(N^*(T)\big)\geq T$. This, combined with \eqref{11.1}, leads to \eqref{2.6}.
 At last, by \eqref{2.6} and \eqref{4.1}, $u^T$ is  the optimal control for $({\rm TP})_{N^*(T)}$.
\par {\itshape Step 2.}~We prove \eqref{2.7}.
\par Let $M\in[0,+\infty)$ be such that $t^*(M)\in(\tau_k,\tau_{k+1})$ for some $1\leq k\leq k_0$. Denote by $u^M$ the optimal control for $({\rm TP})_M$. Combining this with the bang-bang property of $u^M$, we get that
\begin{equation}\label{4.4}
\big\|y\big(t^*(M);y_0,u^M\big)\big\|\leq r~~\text{ and }~~\big\|u^M\big\|_\infty =M.
\end{equation}
This shows that $u^M$ is an admissible control for $({\rm NP})_{t^*(M)}$ and then
\begin{equation*}
N^*\big(t^*(M)\big)\leq \big\|u^M\big\|_\infty=M.
\end{equation*}
If $N^*\big(t^*(M)\big)<M$, then there would exist a control $\widehat{u}$ so that
\begin{equation*}
\big\|y\big(t^*(M);y_0,\widehat{u}\big)\big\|\leq r~~\text{ and }~~\|\widehat{u}\|_\infty=N^*\big(t^*(M)\big)<M.
\end{equation*}
This shows that $\widehat{u}$ is also an optimal control for $({\rm TP})_M$, which contradicts  the bang-bang property of $({\rm TP})_M$. Therefore, \eqref{2.7} stands. This, together with \eqref{4.4}, shows that $u^M$ is also an optimal control for $({\rm NP})_{t^*(M)}$.
\par {\itshape Step 3.}~We prove \eqref{2.8}.
\par We suppose that $t^*(M)=\tau_k$ for some $2\leq k\leq k_0$. Denote by $\widetilde u$ the optimal control with minimal norm for $({\rm TP})_M$. By the similar proof in step 2, we have
\begin{equation*}
N^*\big(t^*(M)\big)\leq M_1.
\end{equation*}
If $N^*\big(t^*(M)\big)< M_1$, there would exist a control $\widehat{u}$ so that
\begin{equation*}\label{0717-3}
\big\|y(t^*(M);y_0,\widehat{u})\big\|\leq r~~\text{ and }~~\|\widehat{u}\|_\infty=N^*\big(t^*(M)\big)<M_1\leq M.
\end{equation*}
This shows that $\widehat{u}$ is also an optimal control for $({\rm TP})_{M}$,
and contradicts the fact that $\widetilde u$ is the optimal control with minimal norm for $({\rm TP})_M$.
Thus \eqref{2.8} stands. This shows that  $\widetilde u$ is an optimal control for $({\rm NP})_{t^*(M)}$.
\par In summary, we finish the proof of Theorem \ref{Theorem 1.2}.

\section{Analytical properties of the minimal time function \texorpdfstring{$t^*(\cdot)$}{t*(·)}}
\par Asssume that (C1) stands.  It is easy to verify that
\begin{equation}\label{0717-4}
t^*(0)=\gamma(y_0) \textrm{ and }\lim_{M\to\infty} t^*(M)=\tau_1.
\end{equation}
By  Lemma \ref{lemma 3.2} and \eqref{2.6},  we have
\begin{equation*}
\big\{M\geq 0:\, t^*(M)= \tau_k\big\} \neq \varnothing  \,\,\textrm{for each}\,\,2\leq k\leq k_0.
\end{equation*}
For each $2\leq k\leq k_0$, we  recall  the definitions
\begin{equation}\label{4.5}
M_k^{\inf}:=\inf\big\{M\geq 0:\, t^*(M)= \tau_k\big\}
\end{equation}
and
\begin{equation}\label{4.6}
M_k^{\sup}:=\sup\big\{M\geq 0:\, t^*(M)= \tau_k\big\}.
\end{equation}
Then by Proposition  \ref{Proposition 2.6} and \eqref{0717-4}, we obtain that
\begin{equation}\label{4.7}
0\leq M_{k_0}^{\inf}\leq M_{k_0}^{\sup}\leq M_{k_0-1}^{\inf}\leq M_{k_0-1}^{\sup} \leq\cdot\cdot\cdot\leq M_{2}^{\inf}\leq M_{2}^{\sup}<+\infty.
\end{equation}
Moreover, for each $2\leq k\leq k_0$, it stands that
\begin{equation}\label{4.8}
t^*(M)=\begin{cases}
\tau_k,~~\text{ if  }M_{k}^{\inf}<M_{k}^{\sup}\textrm{ and }M\in(M_{k}^{\inf},M_{k}^{\sup});\\
\tau_k,~~\text{ if  }M=M_{k}^{\inf}=M_{k}^{\sup}.
\end{cases}
\end{equation}
Furthermore, $M_{k}^{\inf}$ and $M_{k}^{\sup}$ have the following relationship with $N^*(\cdot)$.
\vspace{0.2cm}

\begin{lemma}\label{lemma 4.1}~Assume that (C1) stands. For each $2\leq k\leq k_0$, it stands that
\begin{equation}\label{300}
M_{k}^{\inf}=N^*(\tau_k)~~\text{ and }~~M_{k}^{\sup}=\lim\limits_{t\rightarrow \tau_k^-}N^*(t).
\end{equation}
\end{lemma}

\begin{proof} Fix $2\leq k\leq k_0$. The proof is divided into 2 steps.

{\it Step 1.}~We prove that
\begin{equation}\label{0717-5}
M_{k}^{\inf}=N^*(\tau_k).
\end{equation}
 By \eqref{2.6}, we have that $t^*\left(N^*(\tau_k)\right)=\tau_k$. This, combined with the definitions of $M_k^{\inf}$ and $M_k^{\sup}$, shows that
 \begin{equation}\label{0717-6}
 M_k^{\inf}\leq  N^*(\tau_k)\leq M_k^{\sup}.
 \end{equation}
 By contradiction, suppose that $M_k^{\inf}< N^*(\tau_k)$.
 Then by \eqref{0717-6} and \eqref{4.8}, there would exist $M\in\left(M_k^{\inf},N^*(\tau_k)\right)$
 such that  $t^*(M)=\tau_k$.  Then according to the existence of optimal controls to $({\rm TP})_{M}$,   there exists $u$ satisfying
 $$
 y(t^*(M);y_0, u)\in B_r(0) \textrm{ and }\|u\|_{\infty}\leq M.
 $$
 This, together with the fact that $t^*(M)=\tau_k$, indicates that $u$ is an admissible control for $({\rm NP})_{\tau_k}$
 and then $N^*(\tau_k)\leq M$, which leads to a contradiction. Thus we finish the proof of \eqref{0717-5}.

 {\it Step 2.}~We prove that
 \begin{equation}\label{0717-7}
 M_{k}^{\sup}=\lim\limits_{t\rightarrow \tau_k^-}N^*(t).
 \end{equation}
 \par According to Proposition  \ref{Proposition 3.1},
 there are only two possible cases  for the relationship between $\lim\limits_{t\rightarrow \tau_k^-}N^*(t)$ and $N^*(\tau_k)$:
 either $\lim\limits_{t\rightarrow \tau_k^-}N^*(t)=N^*(\tau_k)$ or $\lim\limits_{t\rightarrow \tau_k^-}N^*(t)>N^*(\tau_k)$.
 Next we prove  \eqref{0717-7} for each case separately.

 \par Case 1.
\begin{equation}\label{400}
\lim\limits_{t\rightarrow \tau_k^-}N^*(t)=N^*(\tau_k).
\end{equation}
\par In this case, from \eqref{4.7} and \eqref{0717-5}  we have $M_{k}^{\sup}\geq \lim\limits_{t\rightarrow \tau_k^-}N^*(t)$. By contradiction, if $M_{k}^{\sup}>\lim\limits_{t\rightarrow \tau_k^-}N^*(t)$,
then
\begin{equation*}\label{500}
M_{k}^{\sup}>M_{k}^{\inf}=N^*(\tau_k).
\end{equation*}
From Proposition  \ref{Proposition 3.2} and \eqref{400}, we know that $N^*(\cdot)$ is continuous and strictly decreasing over $(\tau_{k-1},\tau_k)$. Then for any $\widetilde{M}\in(N^*(
\tau_k),M_{k}^{\sup})$ sufficiently close to $N^*(\tau_k)$, there exists $\tau\in(\tau_{k-1},\tau_k)$, so that $N^*(\tau)=\widetilde{M}$. Then by using Theorem \ref{Theorem 1.2}, we get that
\begin{equation*}
\tau=t^*\big(N^*(\tau)\big)=t^*\big(\widetilde{M}\big).
\end{equation*}
This shows that $t^*\big(\widetilde M\big)<\tau_k$, which, together with the fact that
\begin{equation*}
M_{k}^{\inf}=N^*(\tau_k)<\widetilde{M}<M_{k}^{\sup},
\end{equation*}
leads to a contradiction with \eqref{4.8}.
\par Case 2. $\lim\limits_{t\rightarrow \tau_k^-}N^*(t)>N^*(\tau_k)$.
\par In this case, write $\lim\limits_{t\rightarrow \tau_k^-}N^*(t)=N_0$. First we claim that
\begin{equation}\label{4.11}
M_{k}^{\sup}\leq N_0.
\end{equation}
In fact, by the definition of $N_0$, there exists a sequence $\{T_n\}_{n\geq 1}\subset(\tau_{k-1},\tau_k)$ strictly  increasing to $\tau_k$ such that
\begin{equation*}\label{12.10}
\lim\limits_{n\rightarrow \infty} N^*(T_n)=N_0\,\,\text{ and }\,\,N^*(T_1)>N^*(T_2)>\cdot\cdot\cdot>N^*(T_n)>N_0.
\end{equation*}
Also there exists a sequence $\{M_n\}_{n\geq 1}$ monotonically increasing to $M_{k}^{\sup}$, so that $t^*(M_n)=\tau_k$.
We notice that
\begin{equation}\label{4.12}
M_n<N^*(T_n)~~\text{ for each }~~n\geq 1.
\end{equation}
Otherwise, if $M_n\geq N^*(T_n)$, then by the monotonicity of $t^*(\cdot)$ and Theorem \ref{Theorem 1.2}, we would have that
\begin{equation*}
T_n=t^*\big(N^*(T_n)\big)\geq t^*\big(M_n\big)=\tau_k.
\end{equation*}
This leads to a contradiction with the fact that $T_n\in(\tau_{k-1},\tau_k)$. Passing to the limit for $n\rightarrow\infty$ in \eqref{4.12}, we obtain \eqref{4.11}.
\par Second, by contradiction, suppose
\begin{equation*}
M_{k}^{\sup}=N_0-2\delta~~\text{ for some }~~\delta>0.
\end{equation*}
According to the monotonicity of $t^*(\cdot)$ proved in Proposition  \ref{Proposition 2.6}, we have $t^*\big(N_0-\delta\big)\leq t^*\big(M_{k}^{\sup}\big)$. This, combined with the definition of $M_{k}^{\sup}$ and the fact that $t^*\big(M_{k}^{\sup}\big)=\tau_k$, indicates that
\begin{equation}\label{600}
T_\delta \triangleq t^*\big(N_0-\delta\big)<t^*\big(M_{k}^{\sup}\big)= \tau_k.
\end{equation}
On one hand, by the monotonicity of $N^*(\cdot)$, \eqref{600} yields that
\begin{equation*}
N^*(T_\delta)>N^*(t)\,\,\text{ for any }\,\,t\in(T_\delta,\tau_k).
\end{equation*}
This, together with the fact that $\lim\limits_{t\rightarrow \tau_k^-}N^*(t)=N_0$, shows that
\begin{equation}\label{700}
N^*(T_\delta)\geq N_0.
\end{equation}
On the other hand, by \eqref{600} and the definition of $({\rm TP})_{N_0-\delta}$, there exists a control $\widetilde{u}$ so that
\begin{equation*}
\|\widetilde{u}\|_\infty\leq N_0-\delta~~\text{ and }~~\|y(T_\delta;y_0,\widetilde{u})\|\leq r.
\end{equation*}
This shows that $\widetilde{u}$ is an admissible control for problem $({\rm NP})_{T_\delta}$. Thus
\begin{equation*}\label{4.13}
N^*(T_\delta)\leq N_0-\delta,
\end{equation*}
which contradicts  \eqref{700}.
\par This completes the proof.
\end{proof}

\begin{lemma}\label{lemma 4.2}~Assume that (C1) stands. For each $2\leq k\leq k_0$, it stands that
\begin{equation}\label{0717-9}
t^*(M)=\tau_k \textrm{ for any }M\in\big[M_k^{\inf}, M_k^{\sup}\big].
\end{equation}
\end{lemma}
\begin{proof}
Let $2\leq k\leq k_0$ be fixed. Without loss of generality, we assume that
\begin{equation*}
M_k^{\inf}<M_k^{\sup}.
\end{equation*}

Firstly,  by the first equation in \eqref{300} and \eqref{2.6}, we have
\begin{equation}\label{0717-8}
t^*\big(M_k^{\inf}\big)=t^*\big(N^*(\tau_k)\big)=\tau_k.
\end{equation}

Secondly,
on one hand, by the definition of $M_k^{\sup}$,
there exists a sequence $\{M_n\}_{n\in \mathbb N^+}\subset\big(M_k^{\inf},M_k^{\sup}\big)$
strictly increasing to $M_k^{\sup}$ and
$
t^*(M_n)=\tau_k.
$
This, together with Proposition \ref{Proposition 2.6}, shows that
\begin{equation*}
t^*(M_k^{\sup})\leq \tau_k.
\end{equation*}
On the other hand, let $\{T_n\}_{n\in\mathbb N^+}\subset(\tau_{k-1},\tau_k)$ be a sequence strictly increasing to $\tau_k$.
By Proposition \ref{Proposition 3.1} and the second equation in \eqref{300}  , we have that
\begin{equation*}
N^*(T_1)>N^*(T_2)>\cdots>N^*(T_n)>\cdots>M^{\sup}_k.
\end{equation*}
This, together with Proposition \ref{Proposition 2.6} and (i) of Theorem \ref{Theorem 1.2}, shows that
\begin{equation*}
\tau_k\leq t^*(M_k^{\sup}).
\end{equation*}
Therefore, we have
\begin{equation*}
\tau_k= t^*(M_k^{\sup}).
\end{equation*}
This, together with \eqref{0717-8} and \eqref{4.8}, completes the proof.
\end{proof}

Based on Theorem \ref{Theorem 1.2}, Lemma \ref{lemma 2.2}, Lemma \ref{lemma 3.2},   Corollary \ref{Corollary 2.4} and Lemma  \ref{lemma 4.1},
 we have the following corollary for problem $({\rm NP})_T$.
\begin{corollary}
Assume that (C1) stands. For each $T\in(\tau_1,\gamma(y_0)]$, $({\rm NP})^T$ admits a unique optimal control, denoted by  $u^T=(u^T_j)_{j\in\mathbb N^+}$.
Moreover, the optimal control is bang-bang, that is,
$$\|u^T_j\|=N^*(T)  \textrm{ for each }1\leq j\leq k,$$
where $k$ is the maximum integer such that $\tau_k\leq T$.
\end{corollary}
\begin{proof}
When $T\neq \tau_k$ for any $2\leq k\leq k_0$, the results follow from Lemma \ref{lemma 3.2}, (i) of Theorem \ref{Theorem 1.2} and Lemma \ref{lemma 2.2}.

When $T= \tau_k$ for some $2\leq k\leq k_0$, by   Lemma \ref{lemma 3.2} and  (i) of Theorem \ref{Theorem 1.2},
there exists an  optimal control $u^k$ to $({\rm NP})_{\tau_k}$ with $\|u^k\|_{\infty}=N^*(\tau_k)$, which is
an optimal control for $({\rm TP})_{N^*(\tau_k)}$. Moreover, it stands that
$$
t^*\big(N^*(\tau_k)\big)=\tau_k.
$$
This, together with the definition of $M_k^{\inf}$ and  the first equation in \eqref{300},
shows that $u^k$ is an optimal control with minimal norm for $({\rm TP})_{N^*(\tau_k)}$.
Then the results follow from Corollary \ref{Corollary 2.4}.
\end{proof}

 We  are finally in a position to prove  Theorem \ref{Theorem 4.4}, which is mainly concerned with  the continuity  and monotonicity of the minimal time function $t^*(\cdot)$.

\begin{proof}[Proof of Theorem \ref{Theorem 4.4}]
By the definition of $\gamma(y_0)$ and the assumption (C1), it is easy to verify  that \eqref{0717-15} is true.
For  convenience,  we also denote
$$
t^*(+\infty)=\tau_1, \,N^*(\tau_1)=+\infty \textrm{ and }M_1^{\inf}=+\infty.
$$
Without loss of generality, we suppose that $\gamma(y_0)=\tau_{k_0}$. Then by the definitions of $\gamma(y_0)$ and Lemma \ref{lemma 4.2},
we have that $M_{k_0}^{\inf}=0$. This, together with \eqref{4.7}, shows that
$$
[0,+\infty)\backslash \left(\cup_{k=2}^{k_0}\left[M_{k}^{\inf},M_{k}^{\sup}\right]\right)
= \cup_{k=2}^{k_0}\left(M_{k}^{\sup}, M_{k-1}^{\inf}\right).
$$

Fix  $2\leq k\leq k_0$, we first prove that $t^*(\cdot)$ is continuous and strictly decreasing over $(M_{k}^{\sup}, M_{k-1}^{\inf}]$.
By Theorem \ref{Theorem 1.2} and \eqref{0717-9}, as well as  the definitions of $M_k^{\inf}$ and $M_k^{\sup}$ , we have that
\begin{equation}\label{0717-10}
t^*(N^*(T))=T \textrm{ for  any } T\in [\tau_{k-1}, \tau_{k})
\end{equation}
and
\begin{equation*}\label{0717-11}
N^*(t^*(M))=M \textrm{ for  any } M\in (M_{k}^{\sup}, M_{k-1}^{\inf}].
\end{equation*}
By Proposition \ref{Proposition 3.1}, Proposition \ref{Proposition 3.2} and Lemma \ref{lemma 4.1}, we have that
$N^*(\cdot)$ is continuous and strict decreasing over $[\tau_{k-1},\tau_{k})$. Moreover,
\begin{equation}\label{0717-12}
N^*(\cdot) \textrm{ maps }[\tau_{k-1},\tau_{k}) \textrm{ to }(M_{k}^{\sup}, M_{k-1}^{\inf}].
\end{equation}
By \eqref{0717-10}-\eqref{0717-12}, we know that the functions $t^*(\cdot)$ restricted to $(M_{k}^{\sup}, M_{k-1}^{\inf}]$
and $N^*(\cdot)$ restricted to $[\tau_{k-1},\tau_{k})$ are  inverses of each other.
Therefore the function $t^*(\cdot)$ is continuous and strictly decreasing over $(M_{k}^{\sup}, M_{k-1}^{\inf}]$ for each $2\leq k\leq k_0$.
This, together with Lemma \ref{lemma 4.2}, shows that $t^*(\cdot)$ is continuous over $[0,+\infty)$.
\end{proof}

\begin{remark}\label{Remark 3.5}~Assume that (C1) stands.

 (i)\,\,One may ask if the function ¡°$N^*(\cdot)$¡± is also continuous at each instant $\tau_k$. The following  example shows that  $N^*(\cdot)$ may only be right continuous but not continuous at the impulse instants.
\par Let $e_1$ be the normal eigenfunction corresponding to the first eigenvalue $\lambda_1$ of $-\Delta$ with zero boundary. Set $y_0=e_1$ and $r=\frac{1}{6}$ in  system \eqref{2.1}. Then we have
\begin{equation*}
y(t;y_0,0)=e^{-\lambda_1t}y_0~~\text{ and }~~\gamma(y_0)=\frac{1}{\lambda_1}\ln 6.
\end{equation*}
Let $\tau_1=\frac{1}{\lambda_1}\ln 2$  and $\tau_2=\frac{1}{\lambda_1}\ln 4$ be the two impulse instants in $(0,\gamma(y_0))$. Let $u=(u_1,u_2)$ be the control. Assume that $\omega_1=\omega_2=\Omega$, then we have that
\begin{align*}
y(\tau_2;y_0,u)&=\chi_{\omega_2}u_2+e^{-\lambda_1(\tau_2-\tau_1)}\chi_{\omega_1}u_1+e^{-\lambda_1\tau_2}y_0\\
&=u_2+\frac{1}{2} u_1+\frac{1}{4} e_1
\end{align*}
and
\begin{align*}
y(\tau_2^-;y_0,u)&=e^{-\lambda_1(\tau_2-\tau_1)}\chi_{\omega_1}u_1+e^{-\lambda_1\tau_2}y_0=\frac{1}{2} u_1+\frac{1}{4} e_1.
\end{align*}
By using the bang-bang property of $({\rm NP})_{\tau_2}$, we have that $N^*(\tau_2)=\frac{1}{18}$ and $\lim\limits_{t\rightarrow \tau_2^-}N^*(t)=\frac{1}{6}$. Therefore, $N^*(\cdot)$ is not continuous at $\tau_2$.
\par(ii)\,\,The function $t^*(\cdot)$ need not be strictly decreasing on $[0,+\infty)$. For instance,  in the example provided in (i),  Lemma \ref{lemma 4.1} implies  $M_{2}^{\inf}=\frac{1}{18}$, and $M_{2}^{\sup}=\frac{1}{6}$. Then by Lemma \ref{lemma 4.2},  we have  $t^*(M)=\tau_2$ for any $M\in\left[\frac{1}{18},\frac{1}{6}\right]$.
\par(iii)\,\,Based on the conclusions we have reached, the functions $N^*(\cdot)$ and $t^*(\cdot)$ can be roughly plotted, as shown in  Figure 1 and Figure 2.
\begin{figure}[htbp]
    \centering
    \begin{minipage}{0.46\textwidth}
        \centering
        \includegraphics[width=\textwidth]{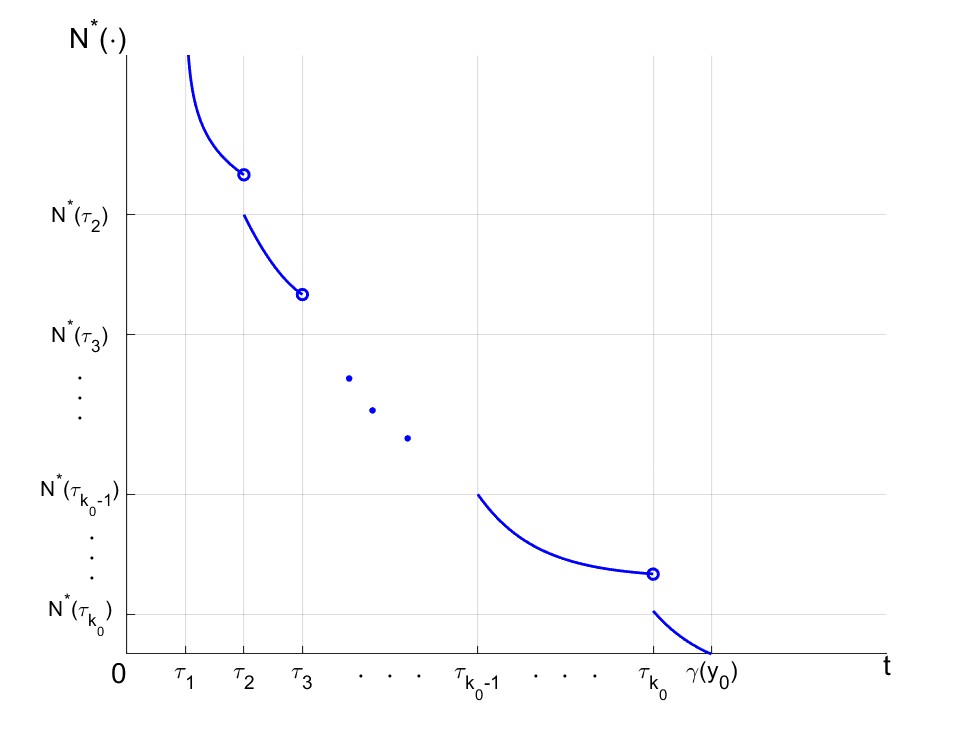}
        \caption{A general plot of the function $N^*(\cdot)$}  
    \end{minipage}
    \hspace{0.04\textwidth}  
    \begin{minipage}{0.46\textwidth}
        \centering
        \includegraphics[width=\textwidth]{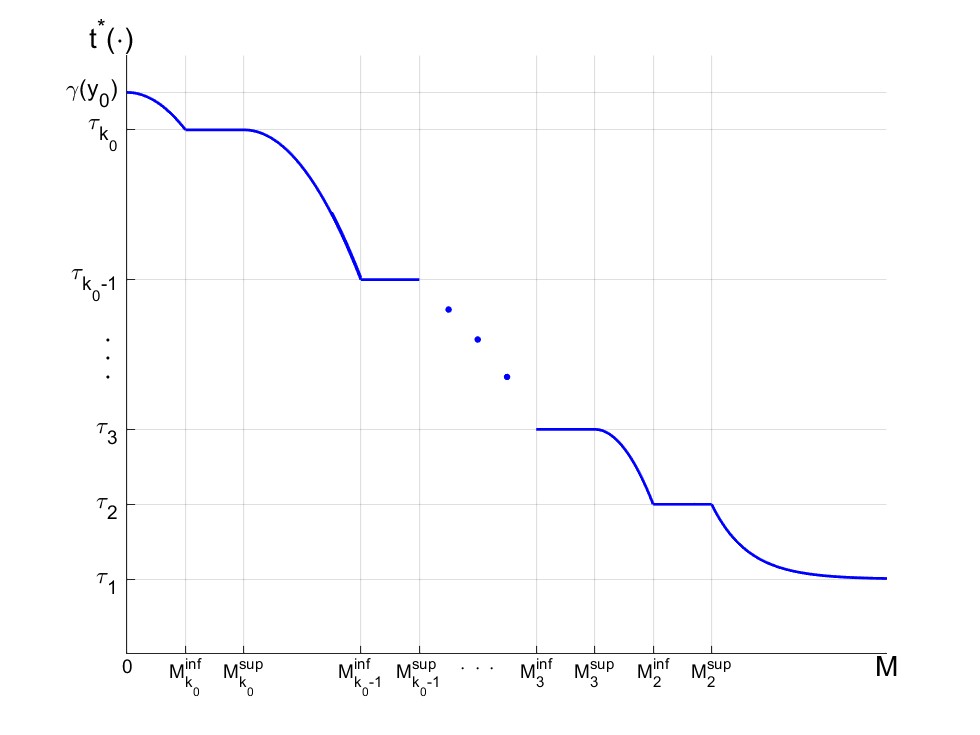}
        \caption{A general plot of the function $t^*(\cdot)$}
    \end{minipage}
\end{figure}
\end{remark}

\begin{remark}\label{Remark 3.6}~The results presented in Section 3 and Section 4 are derived under   condition (C1).
 Analogous results remain valid  under  condition (C2) with only minor modifications.
 Defining $\widehat{M}:=N^*(\tau_1)$, we obtain the following results for this case.
 \begin{itemize}
\item[(i)]
For each $T\in[\tau_1,\gamma(y_0)]$, problem $({\rm NP})_T$ has a unique optimal control and the optimal control is bang-bang.
\item[(ii)] The function $N^*(\cdot)$ is strictly  decreasing over $[\tau_1,\gamma(y_0)]$, right continuous over $[\tau_1,\gamma(y_0)]$ and left continuous over $[\tau_1,\gamma(y_0)]\backslash\{\tau_k\}_{k=1}^{k_0}$.
\item[(iii)]
The function $t^*(\cdot):\,[0,+\infty)\to[\tau_1, \gamma(y_0)]$ is continuous. We can also define $M_k^{\sup}$ and $M_k^{\inf}$ by \eqref{4.5} and \eqref{4.6} for each $2\leq k\leq k_0$, respectively. It keeps still over $[M_{k}^{\inf},M_{k}^{\sup}]$ for each $2\leq k\leq k_0$, and is strictly  decreasing over $\left[0,\widehat M\right)\backslash \left(\cup_{k=2}^{k_0}\left[M_{k}^{\inf},M_{k}^{\sup}\right]\right)$.
 We also know that $t^*(0)=\gamma(y_0)$ and $t^*(M)=\tau_1$ for any $M\geq\widehat{M}$.
\item[(iv)]
The results in Theorem \ref{Theorem 1.2} are still true. Besides, we have that $t^*(N^*(\tau_1))=\tau_1$, and the optimal control for $({\rm NP})_{\tau_1}$
is the optimal control for $({\rm TP})_{\widehat M}$.   For any $M\geq \widehat M$,   $N^*(t^*(M))=\widehat M$, and the optimal control with minimal norm for
$({\rm TP})_M$ is the optimal control for $({\rm NP})_{\tau_1}$.
\end{itemize}
\end{remark}

\section{Appendix}
\begin{proof}[Proof of Lemma \ref{Theorem 2.1}]
It is easy to show that
\begin{center}
$0$ is an admissible control for problem $({\rm TP})_M$.
\end{center}
In fact, it stands that
\begin{equation*}\label{3.1}
\|y(t;y_0,0)\|=\|e^{\Delta t}y_0\|\leq e^{-\lambda_1 t}\|y_0\|,
\end{equation*}
where $\lambda_1$ is the first eigenvalue of $-\Delta$ with homogeneous Dirichlet boundary on $\partial\Omega$. Then we obtain that
\begin{equation*}\label{3.31}
\|y(t;y_0,0)\|\leq r\,\,\text{ when }\,\,t=\tau_1+\frac{1}{\lambda_1}\ln\frac{2r+\|y_0\|}{r},
\end{equation*}
which indicates that $y(t;y_0,0)\in B_r(0)$. Hence, $0$ is an admissible control for problem $({\rm TP})_M$.
\par Suppose that $t^*(M)\in[\tau_k,\tau_{k+1})$ for some $k\in\mathbb{N}^+$. We next prove that $({\rm TP})_M$ has at least one optimal control. According to the existence of admissible controls, there exist two sequences $\big\{T_n\big\}_{n\geq1}\subseteq[\tau_k,\tau_{k+1})$ and $\big\{u^n\big\}_{n\geq1}\subseteq\mathcal{U}_M$ so that
\begin{equation}\label{3.2}
T_n\rightarrow t^*(M)\in[\tau_k,\tau_{k+1})~~\text{and}~~\big\{y(T_n;y_0,u^n)\big\}_{n\geq1}\subseteq B_r(0),
\end{equation}
where $u^n=(u^{n}_1,u^{n}_2,...,u^{n}_k,\,0,\,0,...)$. Since $\{u^{n}\}_{n\geq 1}\subseteq \mathcal{U}_M$, there exists a subsequence of $\{n\}_{n\geq1}$, still denoted by itself, and a control $\widehat{u}=\left(\widehat{u}_1,\widehat{u}_2,...,\widehat{u}_k,\,0,\,0,...\right)\in \mathcal{U}_M$, so that
\begin{equation}\label{3.3}
u^{n}_j\,\xrightarrow{w}\,\widehat{u}_j ~~\text{in}~~ L^2(\Omega)~~\text{for each}~~1\leq j\leq k.
\end{equation}
Denote $y^n(\cdot):=y(\cdot;y_0,u^n)$~~and~~$\widehat{y}(\cdot):=y(\cdot;y_0,\widehat{u})$. Then, we have
\begin{equation}\label{3.101}
y^n(t)=e^{\Delta t}y_0+\sum\limits_{j=1}^ke^{\Delta(t-\tau_j)}\chi_{\omega_j}u^{n}_j~~\text{ for any }~~t\in[\tau_k,\tau_{k+1})
\end{equation}
and
\begin{equation*}
\widehat{y}(t)=e^{\Delta t}y_0+\sum\limits_{j=1}^ke^{\Delta(t-\tau_j)}\chi_{\omega_j}\widehat{u}_j~~\text{ for any }~~t\in[\tau_k,\tau_{k+1}).
\end{equation*}
This, combined with \eqref{3.3}, shows that
\begin{equation}\label{3.4}
y^n(t){\,\xrightarrow{w}\,}\widehat{y}(t)~~\text{in}~~L^2(\Omega)~\text{ for any }~t\in[\tau_k,\tau_{k+1}).
\end{equation}
Since \eqref{3.101} stands and $\|u_j^n\|\leq M$ for each $1\leq j\leq k$, according to the first relation of \eqref{3.2}, we can easily get that
\begin{equation}\label{3.60}
y^n(T_n)\rightarrow y^n\big(t^*(M)\big)\,\,\text{ as }\,\,n\rightarrow \infty.
\end{equation}
Since
\begin{equation*}
y^n(T_n)-\widehat{y}\big(t^*(M)\big)=y^n(T_n)-y^n\big(t^*(M)\big)+y^n\big(t^*(M)\big)-\widehat{y}\big(t^*(M)\big),
\end{equation*}
by \eqref{3.4} and \eqref{3.60}, we can directly check that
\begin{equation*}\label{3.7}
y\big(T_n;y_0,u^n\big){\,\xrightarrow{w}\,\,}y\big(t^*(M);y_0,\widehat{u}\big)~~\text{in}~~L^2(\Omega).
\end{equation*}
This, together with the second relation of \eqref{3.2}, implies that $y\big(t^*(M);y_0,\widehat{u}\big)\in B_r(0)$. Thus, $\widehat{u}$ is an optimal control for problem $({\rm TP})_M$.

The existence of an optimal control with minimal norm
then follows by considering a minimizing sequence for the norm within the set of optimal controls (which is nonempty, as just proved)
and applying a similar weak convergence argument.

\end{proof}


\begin{thebibliography}{00}
\bigskip
{\rm
\bibitem{42}
Arada N. and Raymond J. P., Time optimal problems with Dirichlet boundary controls, Discrete Contin. Dyn. Syst., 2003, 9(6): 1549-1570.
\bibitem{03}
Barbu V., Optimal Control of Variational Inequalities, Pitman, London, 1984.

\bibitem{02}
Barbu V., Analysis and Control of Nonlinear Infinite Dimensional Systems, Academic Press, Boston, 1993.

\bibitem{40}
Bellman R., Glicksberg I. and Gross O., On the \textquotedblleft bang-bang\textquotedblright control problem, Quart. Appl. Math., 1956, 14(1): 11-18.

\bibitem{04}
Brezis H., Function Analysis, Sobolev Spaces and Partial Differential Equations, Springer, New York, 2011.

\bibitem{05}
Chen N., Wang Y. and  Yang D., Time-varying bang-bang property of time optimal controls for heat equation and its applications, Syst. Control Lett., 2018, 112: 18-23.
\bibitem{08}
Duan Y. and Wang L., Minimal norm control problem governed by semilinear heat equation with impulse control, J. Optim. Theory Appl., 2020, 184 (2): 400-418.
\bibitem{09}
Duan Y., Wang L. and Zhang C., Minimal time impulse control of an evolution equation, J. Optim. Theory Appl., 2019, 183 (3): 902-919.
\bibitem{47}
Fattorini H.O., Infinite Dimensional Linear Control Systems: The Time Optimal and Norm Optimal Problems, Elsevier, Amsterdam, 2005.
\bibitem{38}
Gozzi F. and Loreti P., Regularity of the minimum time function and minimum energy problems: the linear case, SIAM J. Control Optim., 1999, 37(4): 1195-1221.

\bibitem{Huang}
Huang J. and Yu  X., Error estimates of semidiscrete finite element approximation for minimal time impulse control problems,
Discrete  Contin. Dyn. Syst. Ser. B, 2024, 29(10): 4345-4360.

\bibitem{14}
Kunisch K. and  Wang L., Time optimal control of the heat equation with pointwise control constraints, ESIAM Control Optim. Calc. Var., 2013, 19(2): 460-485.

\bibitem{17}
Lakshmikantham V. and Simeonov P., Theory of Impulsive Differential Equations, World Scientific, Singapore, 1989.

\bibitem{18}
Li X., Bohner M. and Wang C., Impulsive differential equations: Periodic solutions and applications, Automatica, 2015, 52: 173-178.

\bibitem{46}
Li X. and  Yong J., Optimal Control Theory for Infinite Dimensional Systems, 1995, Birkha\"{u}ser, Boston.

\bibitem{Lin}
Lin G. and Zeng B., Minimal time impulse control for semilinear evolution equations with applications,
Bull. Malays. Math. Sci. Soc., 2025, 48 (4): 119.

\bibitem{16}
Lions J., Optimal Control of System Governed by Partial Differential Equations, Springer, New York, 1971.

\bibitem{44}
Loh\'{e}ac J. and  Tucsnak M., Maximum principle and bang-bang property of time optimal controls for Schr\"{o}dinger-type systems, SIAM J. Control Optim., 2013, 51(5): 4016-4038.
\bibitem{48}
L{u} X., Wang L. and  Yan Q., Computation of time optimal control problems governed by linear ordinary differential equations, J. Sci. Comput., 2017, 73(1): 1-25.



\bibitem{21}
Phung K.D., Wang G. and  Xu Y., Impulse output rapid stabilization for heat equations, J. Differential  Equations, 2017, 263(8): 5012-5041.

\bibitem{22}
Phung K.D., Wang L. and  Zhang C., Bang-bang property for time optimal control of semilinear heat equation, Ann.  I. H. Poincar\'{e} (C) Non Linear Analysis, 2014, 31(3): 477-499.

\bibitem {41}
Pontryagin L.S., Boltyanski V.G., Gamkrelidz R.V. and  Mischenko E.F., The Mathematical Theory of Optimal Processes, Wiley-Interscience, New York, 1962.

\bibitem{23}
Qin S. and Wang G., Controllability of impulse controlled systems of heat equations coupled by constant matrices, J. Differential. Equations, 2017, 263(10): 6456-6493.

\bibitem{49}
Qin S. and  Wang G., Equivalence between minimal time and minimal norm control problems for the heat
equation, SIAM J. Control Optim., 2018, 56(2): 981-1010.
\bibitem{45}
Raymond J.P. and  Zidani H., Pontryagin's prinicple for time-optimal problems, J. Optim. Theory Appl., 1999, 101(2): 375-402.
\bibitem{24}
Samoilenko A. and  Perestyuk N., Impulsive Differential Equations, World Scientific, Singapore, 1995.

\bibitem{26}
Tr\'{e}lat E., Wang L. and  Zhang Y., Impulse and sampled-data optimal control of heat equations, and error estimates, SIAM J. Control Optim., 2016, 54(5): 2787-2819.

\bibitem{25}
Tucsnak M., Wang G. and  Wu C., Perturbations of time optimal control problems for a class of abstract parabolic systems, SIAM J. Control Optim., 2016, 54(6): 2965-2991.
\bibitem{39}
Wang G., Wang L., Xu Y. and  Zhang Y., Time Optimal Control of Evolution Equations, Birkh\"{a}user, Boston, 2018.

\bibitem{28}
Wang G., Xu Y. and  Zhang Y., Attainable subspaces and the bang-bang property of time optimal controls for heat equations, SIAM J. Control Optim., 2015, 53(2): 592-612.

\bibitem{37}
Wang G., Yang D. and  Zhang Y., Time optimal sampled-data controls for the heat equation, C. R. Acad. Sci. Paris, Ser. I, 2017, 35(5): 1252-1290.

\bibitem{51}
Wang G. and  Zuazua E., On the equivalence of minimal time and minimal norm controls for internally controlled heat equations, SIAM J. Control Optim., 2012, 50(5): 2938-2958.

\bibitem{29}
Wang L., Minimal time impulse control problem of semilinear heat equation, J. Optim. Theory Appl., 2021, 188(3): 805-822.

\bibitem{30}
Wang L., Yan Q. and Yu H., Constrained approximate null controllability of coupled heat equation with periodic impulse controls, SIAM J. Control Optim., 2021, 59(5): 3418-3446.

\bibitem{32}
Yang T., Impulsive Control Theory, Springer, Berlin, 2001.

\bibitem{34}
Yang X., Peng D., Lv X. and Li X., Recent progress in impulsive control systems, Math. Comput. Simul., 2019, 155: 244-268.
\bibitem{50}
Yu H., Equivalence of minimal time and minimal norm control problems for semilinear heat equations, Syst. Control Lett., 2014, 73: 17-24.

\bibitem{35}
Zhang C., The time optimal control with constraints of the rectangular type for linear time-varying ODEs, SIAM J. Control Optim., 2013, 51(2): 1528-1542.







%
%










}
\end{thebibliography}
\end{document}